\documentclass[smallextended,referee,envcountsect]{svjour3}

\smartqed
\usepackage{booktabs}
\usepackage{placeins}
\usepackage{algorithm}
\usepackage[misc]{ifsym}
\usepackage{algpseudocode}
\usepackage{subcaption}
\usepackage{array}    
\usepackage{booktabs} 
\usepackage{amsmath}
\usepackage{amsfonts}
\usepackage{graphicx}
\usepackage{graphicx}
\usepackage{tcolorbox,enumerate}
\usepackage{amssymb,latexsym,amsmath,enumitem,amscd,amsthm,graphicx,color}
\usepackage{hyperref}
\allowdisplaybreaks
\newtheorem{Algorithm}{Algorithm}[section]
\newtheorem{e1}{Example}
\begin{document}

\title{An Adaptive Order Caputo Fractional Gradient Descent Method for Multi-objective Optimization Problems}

\author{Barsha Shaw, Md Abu Talhamainuddin Ansary}

\institute{Barsha Shaw and  Md Abu Talhamainuddin Ansary \at
Department of Mathematics\\
IIT Jodhpur, Jodhpur, India 342030\\
Email:\\
B. Shaw:shawbarsha307@gmail.com\\
M. A. T, Ansary: md.abutalha2009@gmail.com (\Letter)
 }


\maketitle
\begin{abstract}
This article introduces the multi-objective adaptive order Caputo fractional gradient descent (MOAOCFGD) algorithm for solving unconstrained multi-objective problems. The proposed method performs equally well for both smooth and non-smooth multi-objective optimization problems. Moreover, the proposed method does not require any a priori chosen parameters or ordering information of the objective functions. At every iteration of the proposed method, a subproblem is solved to identify a suitable descent direction toward an optimal solution. This subproblem involves an adaptive-order Caputo fractional gradient for each objective function. An Armijo-type line search is applied to determine a suitable step length. The convergence of this method for the Tikhonov-regularized solution is justified under mild assumptions. The proposed method is verified using different numerical problems, including neural networks.
\end{abstract}
\section{Introduction}\label{sec1}
In the fractional calculus part, the concept of derivative of order $n^{th}$ ( $n\in\mathbb{N}$) is extended to some fractional order. The study of fractional calculus has become an area of research interest due to its applications in different fields of engineering, such as mechanical, electrical, signal analysis, quantum mechanical, bioengineering, biomedicine, financial systems, machine learning, neural networks, etc. (see \cite{A1,A2,A3,A4}). The Riemann-Liouville (R-L) system fractional derivative ($_aD^{\alpha}_x$) and the Caputo fractional derivative ( $_a^cD^{\alpha}_x$) are two widely used variants of the fractional calculus (see \cite{20}). In \cite{20}, the monotonicity of both fractional derivatives is discussed in $(0,1)$, which is further extended by \cite{Barsha} to any generalized intervals $(\beta,\beta+1)$ where $\beta$ is a positive integer. Some theoretical results for fractional derivatives are discussed in \cite{23,24}. Similarly to integer-order derivatives, different types of fractional-order derivatives have a big impact on optimization. The necessary condition for optimality using the Caputo fractional derivative is developed in \cite{21}. In \cite{38}, a quasi-fractional gradient of a multivariate function is introduced. The Taylor-Riemann series and mean value theorem using the Caputo fractional derivative are introduced in \cite{72}. The gradient descent method using the Caputo fractional derivative was developed there. This derivation uses the Taylor series expansion of a multi-variate function using the Caputo fractional derivative. Convergence of the proposed method is justified for Tikhonov regularization, and numerical experiments are performed for neural network problems.
\par Multi-objective optimization involves minimizing several objective functions as a standard. Applications of the type of problems can be found in computer science engineering, healthcare logistics, supply chain management, machine learning, neural architecture search, financial portfolio optimization, electrical signal analysis, quantum mechanics, bio-engineering, bio-medicine, financial systems, and machine learning, neural networks, etc (see \cite{45,46,44,43}). Older methods for solving multi-objective problems use either scalarization or heuristic approaches. However, scalarization techniques are user-dependent, and heuristic approaches can not guarantee the convergence to the solution (see \cite{50}). To overcome these limitations, different descent methods have been developed by different researchers using single-objective descent methods. These techniques include gradient-based techniques for smooth problems (see \cite{sir1,67,66,JOTA2,60,13,25}) and proximal methods for non-smooth problems (see \cite{30,61,JOTA1,64}), and descent methods for uncertain multi-objective optimization problems involving finite uncertainty (see \cite{kumar1,kumar2,kumar3}).

Designing a neural network is not just about getting high accuracy; it also involves making sure the model is efficient in terms of size, speed, and energy use. Because there are so many possible combinations of layers, neurons, and activation functions, it is difficult and time-consuming to find the best settings manually. To solve this, \cite{69} proposed a method using multi-objective optimization to design neural networks automatically. Their technique builds a model to predict which combinations are likely to perform well, so fewer networks need to be trained and tested. This idea is also supported by other studies (see \cite{71}), which show that automated neural architecture search can save both time and computing resources. This article proposes a descent method for solving multi-objective optimization problems using the Caputo fractional derivative. The proposed methodology incorporates ideas from single-objective optimization. These techniques are applicable to both smooth and non-smooth multi-objective problems. It is also able to give information about non-integer order derivative information of a function.\\

The structure of the article is as follows. The necessary preliminaries for the theoretical derivations are presented in Section~\ref{sec2}. The proposed adaptive order Caputo fractional gradient descent method is developed in Section~\ref{sec3}. An algorithm is proposed in this section. In Section~\ref{sec4}, convergence analysis of the proposed method is justified for quadratic Tikhonov regularization problems. Some numerical examples, including neural networks, are given in Section~\ref{sec5} to demonstrate the method's effectiveness. Finally, the paper is concluded in Section~\ref{sec6}, where possible directions for future work are discussed.
\section{Preliminaries}\label{sec2}
 An unconstrained multi-objective optimization problem is defined as
\begin{equation*}
(MOP): \quad \underset{x\in \mathbb{R}^n}{\min} ~~(f_1(x), f_2(x), \dots, f_m(x)),
\end{equation*}
where $f_j:\mathbb{R}^n \rightarrow\mathbb{R}$ is a continuous function for $j \in \{1,2,\dots,m\}$.\\
Define $\mathbb{R}^m_{+}:=\{y\in \mathbb{R}^m:~y_j\geq 0~~\forall j\}$ and $\mathbb{R}^m_{++}:=\{y\in \mathbb{R}^m:~y_j> 0~~\forall j\}$. For $y^1,y^2\in\mathbb{R}^m$, we say $y^1\leq y^2$ iff $y^1_j\leq y^2_j$ for all $j$ and  $y^1_j< y^2_j$ for at least one $j$. Similarly, $y^1<y^2$ iff $y^1_j <y^2_j$ holds for all $j$. A feasible point $x^*$ is said to be an efficient solution or Pareto optimal solution of $f$ if there does not exist any $x$ such that $f(x)\leq f(x^*)$ holds where, $f(x)=(f_1(x), f_2(x), \dots, f_m(x))$. A feasible point $x^*$ is considered a weakly efficient solution if there does not exist any $x$ such that $f(x)<f(x^*)$ holds. Here, each $f$ is differentiable at the point $x^*$. If $x^*$ is a weak efficient solution,
\begin{equation}
   \underset{j\in\{1,2,\dots,m\}}{\max}~~ f_j^{\prime}(x,d)\geq 0 \label{crit1}
\end{equation}
holds for every $d\in\mathbb{R}^n$. In the literature, the above inequality is considered a necessary condition for weak efficiency. Any point $x^*$ that meets this inequality is a critical point of $(MOP)$. A critical point $x^*$ is a weak efficient solution of $(MOP)$ if every function $f_j$ is a convex function and $x^*$ is an efficient solution if every  $f_j$ is a strictly convex function.
 


\par Descent methods for $(MOP)$ developed so far use the classical derivatives, where the order of differentiation is restricted to integers. In contrast, fractional derivatives generalize this concept by allowing the order to take non-integer values. Among the commonly used formulations, the R–L and the Caputo derivatives are particularly prominent. We denote the set of $n$-th differentiable and $n+1$-th integrable functions on $[a,b]$ by $A^n[a,b]$ as absolutely continuous. For $f_1\in A^n[a,b]$, the  $\alpha$-order the Caputo fractional derivative for $x\in[a,b]$ is formed as,
  $$^C_cD^{\alpha}_xf_1(x) = \frac{1}{\Gamma(n-\alpha)}  \int_{c}^{x}(x-\tau) ^{n-\alpha-1} f_1^{(n)}(\tau)d\tau,$$
here, $\alpha\in(n-1,n), ~n\in  \mathbb{N}$ and $\tau$ be any variable in interval $[a,b]$. A multi-variate continuous function $f_1:\mathbb{R}^n \rightarrow\mathbb{R}$ fractional gradient denote by $^C_c\nabla^{\alpha}_xf_1(x)$ is defined as (\cite{38}),
 \begin{equation*}
     ^C_c\nabla^{\alpha}_xf_1(x) =\Big(~
         ^C_cD^{\alpha}_{x_1}f_1(x), ~
         ^C_cD^{\alpha}_{x_2}f_1(x), \dots,~ ^C_cD^{\alpha}_{x_n}f_1(x)\Big)^T.
 \end{equation*}   
Usually, the Caputo fractional derivative is employed to handle non-differentiable functions where classical derivatives fail. However, only fractionally, the Caputo derivative fully satisfies the necessary optimality conditions for critical purposes solutions. This can be justified by the example.
\begin{example}
Consider an optimization problem,
$$~ \underset{x\in\mathbb{R}^2}{\min} f(x)= 4x_1^2 + x_2^2 - 2x_1 x_2 -3 x_1 + 4x_2.
$$
For $\alpha=0.5$, $ ^C_c\nabla^{\alpha}_xf(x^*)=0$ implies $x^*=(0.34313689, -0.12745096)$ with $f(x^*)=-0.769607$. Clearly, this differs from the critical point via the \\first-order derivative $x^*=(-0.166785, -2.166603)$ with $f_1(x^*)=-4.083333$.
\end{example}
The standard Taylor series based on the Caputo fractional derivative does not provide sufficiently accurate optimality conditions for our problem. Therefore, a modified Taylor series formulation is needed that is based on the Caputo fractional derivative. This modification is particularly important when dealing with multi-variate functions as follows.
\begin{lemma}\cite{72}\label{thm1}
Let $ f: \mathbb{R}^n\rightarrow \mathbb{R} $ admit the following Taylor expansion around $c\in\mathbb{R}^n$ for some $ \alpha \in (0,1) $ and $ \beta \in \mathbb{R} $. 
\[
{}_cf_{1,\alpha,\beta}(x) = f_1(c) + \nabla f_1(c)^T (x - c) + \frac{1}{2} C_{2,\alpha,\beta}~ (x-c)^T H_1(c)(x - c)
\]
where $C_{2,\alpha,\beta} = \frac{\gamma(2-\alpha)\gamma(2)}{\gamma(3-\alpha)} + \beta $ and $H_1$ is  the Hessian of function $f_1$ based on standard derivative. Then the fractional gradient $ {}^C_c\nabla^{\alpha}_x~{}f_{1,\alpha,\beta}(x) $ at $ x $ is given by
\[
{}^C_c\nabla^{\alpha}_x~{}f_{1,\alpha,\beta}(x) 
= diag\Big({}^C_c\nabla^{\alpha}_x I(x)\Big)^{-1}
\Big[ {}^C_c\nabla^{\alpha}_x f_1(x) + \beta\, diag(|x - c|)\cdot {}^C_c\nabla^{1+\alpha}_x f_1(x)\Big],
\]
where $ I : \mathbb{R}^n \to \mathbb{R} $ is the identity map.
\end{lemma}
To address the modified Taylor series, a single-objective optimization approach is modified to the history-dependent behaviour of the fractional derivative. An intuitive method involves replacing the constant lower integration terminal $c$ with a variable lower integration terminal $x^{k-L}$, referred to as an adaptive terminal (see \cite{50}). In this article, an adaptive order Caputo fractional gradient descent(AOCFGD) method for single-objective optimization is developed, which incorporates a sequence of derivative orders along with a varying lower integral terminal. The AOCFGD approach demonstrates improved ability to accurately identify critical points and the corresponding function values, aligning closely with those obtained using standard integer-order derivatives. 
 However, consider AOCFGD, where the fractional order $\alpha_k$ and the smoothing parameter $\beta_k$ are chosen adaptively. The vector $c^k$ remains fixed as $c \in \mathbb{R}^d$ for all $k$.

Consider a sequence $\{(\alpha_s, \beta_s)\}_{s \geq 1}$ where 
$\alpha_s \in(0,1]$ and $\beta_s $  belongs to $ \left[\frac{1 - \alpha_s}{2 - \alpha_s}, \infty \right)$, for all $ s \geq 1$
. Define a corresponding non-negative sequence ${\gamma_s} \geq 1$ by \begin{equation*} \label{eq:gamma_s} \gamma_s = \beta_s - \frac{1 - \alpha_s}{2 - \alpha_s}. \end{equation*}

Consider ${k_s} $, where ${s \geq 1}$ denotes a sequence of positive integers. For a given positive integer $s$, the AOCFGD procedure proceeds in $s$ stages. The first stage begins with an initial point $x^{0}$ and applies the CFGD method described in \cite{72} using parameters $\alpha_1, \beta_1$, and $c$ for $k_1$ iterations. The output at the end of this stage is denoted by $x^{(k_1)}$. For each subsequent stage $s \geq 2$, the algorithm continues from $x_{s-1}^{(k{s-1})}$ and performs $k_s$ iterations using the updated parameters $\alpha_s, \beta_s$, and the same $c$. The final output after stage $s$ is represented by $x_s^{k_s}$. In this article, we have adopted these ideas and developed a descent method for $(MOP)$. Analysis of the convergence of the proposed method is justified for multi-objective Tikhonov regularisation solutions. For details of single-objective Tikhonov regularization readers may see \cite{72}. There is another example of choosing AOCFGD for article.
 \begin{example}Consider an minimisation problem, $f(x)$ where $$f_1(x)=x_1^2+ 4 x_2^2-2x_1 x_2 - 3 x_1 + 4x_2$$
 $\text{at point}~ (1,1)~ \text{and lower bound}~(0,0).$  
 
     For non-adaptive terminal method, at alpha is $0.9$ optimal function is $-0.8931$ at point $x=(2.4531, -0.1666)$. 
      adaptive terminal method, at alpha is $0.9$ and lower bound of iteration $L=1$ optimal function is $18.66$ at point $x=(-2.66, 0.333)$.
      By using AOCFGD method, alpha sequence is $\{0.5,0.7,0.9\}$. Optimal point is $-2.33$ at point $(1.333, -1.664)$.
 \end{example}
Next example highlights the computational advantage of the AOCFGD method over traditional subgradient techniques, especially for nonsmooth optimization problems. 
\begin{example}consider the minimising nonsmooth optimisation problem $f(x)$,
{where} \quad $f(x) = \max\{5x_1 + x_2,\ x_1^2 + x_2^2\}$.
\end{example}
At the point \( (3, 3) \), both components of the objective function are equal, indicating a point of non-differentiability. Using the subgradient method, the descent direction can be computed as a convex combination. While the subgradient method successfully converges to the optimal point, it requires 26 iterations and a total CPU time of 0.0119 seconds. In contrast, the proposed AOCFGD method demonstrates significantly improved efficiency. Employing a sequence of fractional orders \( \alpha = \{0.5, 0.7, 0.9\} \) over three stages with iteration counts \( (50,\ 50,\ 100) \) and a step size parameter \( \beta = 0.01 \), the algorithm converges to the same optimal solution in only 7 iterations, with a CPU time of 0.0023 seconds.
\section{An AOCFGD method based for multi-objective optimization problems}\label{sec3}
In this section, an AOCFGD method is developed for $(MOP)$. In this method, a descent sequence is generated with some initial approximation. At any iterating point $x$, the following subproblem is solved to find a descent direction at $x$.
\begin{equation*}\label{eq3.1}
\begin{aligned}
P_1(x): \quad 
\min_{d} \max_{j} \; & {}^C_c\nabla^{\alpha}_x f_{j,_{\alpha,\beta}}(x)^Td
+ \frac{1}{2} d^T d.
\end{aligned}
\end{equation*}
One can observe that $P_1(x)$ uses an AOCFGD for each objective function at $x$.  This formulation is inspired by Newton-type methods for $(MOP)$ in \cite{3} and the single-objective AOCFGD method in \cite{72}. Subproblem $P_1(x)$ can be written equivalently as
 \begin{equation*}\label{eq3.4}
\begin{aligned}
SP_1(x):\quad\min_{(t,d)\in\mathbb{R}\times \mathbb{R}^n} \quad &  t + \frac{1}{2} d^T d, \\
\text{s.t.} \quad &{}^C_c\nabla^{\alpha}_x f_{j,_{\alpha,\beta}}(x)^T d - t \leq 0,  \quad j \in \{1, 2, \dots, m\} .
\end{aligned}
\end{equation*}
Since $SP_1(x)$ is a convex optimization problem that satisfies the Slater constraint qualification, as for $t=1$, $d=0$, inequalities in $SP_1(x)$ become strict inequalities. Hence, the optimal solution ($t(x),d(x)$) of $SP_1(x)$ satisfies the following KKT Optimality conditions,
\begin{align}
&\sum_{j=1}^m \lambda_j = 1,\label{kkt1}\\
&\sum_{j=0}^m \lambda_j \, {}^C_c\nabla^{\alpha}_x f_{j,_{\alpha,\beta}}(x) + d(x) = 0, \label{eq3.5} \\
&\lambda_j \geq 0, ~\lambda_j \left( {}^C_c\nabla^{\alpha}_x f_{j,_{\alpha,\beta}}(x)^T d(x) - t(x) \right) = 0 \quad \forall j \in \{1, 2, \dots, m\} \label{kkt3}\\
& {}^C_c\nabla^{\alpha}_x f_{j,_{\alpha,\beta}}(x)^T d(x) \leq t(x)~~\forall j \in \{1, 2, \dots, m\}. \label{feas_kkt}
\end{align}
From (\ref{eq3.5}),
\begin{equation}\label{d}
d(x) = -\sum_{j=0}^m \lambda_j
~{}^C_c\nabla^{\alpha}_x f_{j,_{\alpha,\beta}}(x).
\end{equation}
Suppose $(t(x),d(x))$ be the optimal solution of $SP_1(x)$ for some $x\in\mathbb{R}^n$. Then clearly $t(x)+\frac{1}{2}d(x)^Td(x) \leq 0.$\\ Further if $d(x)\neq 0$ then $t(x)\leq -\frac{1}{2} \|d(x)\|^2<0.$ In the following lemma we show that, if $x$ is a critical point then $d(x)=0$ and vice-versa.  
\begin{lemma}\label{lemma 3.1}
Suppose $f_j$ is a continuous for all $j$ and $(t(x),d(x))$ is optimal solution of $SP_1(x)$. Then $x\in \mathbb{R}^n$ is a critical point of $(MOP)$ if and only if $d(x)=0$.   
\end{lemma}
\begin{proof}
If possible, suppose $x$ is a critical point of the $(MOP)$ and let $d(x)\neq 0$. Then from (\ref{feas_kkt}),
\begin{equation*}
  {}^C_c\nabla^{\alpha}_x f_{j,_{\alpha,\beta}}(x)^T d(x) \leq t(x) \leq  -\frac{1}{2}d(x)^Td(x) < 0 
\end{equation*}
 Then using Taylor series expansion (used in Lemma \ref{thm1}) for $x+\eta d(x)$, where $\eta>0 $ sufficient small, we have
 \begin{equation*}
    f_j(x+\eta d(x)) = f_j(x) + \eta ~^C_c\nabla^{\alpha}_x f_{j,_{\alpha,\beta}}(x)^T~d(x)+ \frac{1}{2} C_{2,\alpha,\beta} ~ \eta d(x)^T H_j(c)\eta d(x),
\end{equation*} for all $j$, this implies
 \begin{equation*}
    \frac{f_j(x+\eta d(x))-f_j(x)}{\eta}= ~^C_c\nabla^{\alpha}_x f_{j,_{\alpha,\beta}}(x)^T~d(x)+  \frac{\eta}{2}~C_{2,\alpha,\beta} d(x)^T H_j(c) d(x),
\end{equation*}
 taking limit $\eta \rightarrow 0^{+}$, we have $^C_c\nabla^{\alpha}_x f_{j,_{\alpha,\beta}}(x)^T~d(x)<0$ for all $j$. This contradicts (\ref{crit1}). Hence if $x$ is a critical point then $d(x)=0$.\\

Conversely, suppose $d(x)=0$ (see \ref{d}) holds for some $x$. Then from (\ref{kkt1}) and (\ref{eq3.5}), there exists $\lambda\in\mathbb{R}^m$ such that 
\begin{eqnarray*}
\sum_{j=0}^{m} \lambda_j=1 ~ \mbox{ and } \sum_{j=0}^m \lambda_j ~{}^C_c\nabla^{\alpha}_x f_{j,_{\alpha,\beta}}(x)=0 \label{gd1}
\end{eqnarray*}
hold. If possible, suppose $x$ is not a critical point. Then from (\ref{crit1}), there exists some $\bar{d}\in\mathbb{R}^n$ such that $f_j^{\prime}(x,\bar{d})<0$ holds for every $j$. This implies 
\begin{align*} 
\underset{\xi \rightarrow 0^+}{\lim} \frac{f_j(x+\xi \bar{d})- f_j(x)}{\xi}<0,
\end{align*}
using modified Taylor series expansion of $f_j(x+\xi \bar{d})$,
\begin{align*} 
lim_{\xi\rightarrow 0^+} \frac{\xi~^C_c\nabla^{\alpha}_x f_{j,_{\alpha,\beta}}(x)^T~\bar{d}+ \xi^2 \frac{1}{2} \bar{d}^T H_j \bar{d}}{\xi}<0,
\end{align*}
this implies $$^C_c\nabla^{\alpha}_xf_{j,_{\alpha,\beta}}(x)^T~\bar{d}<0~\forall j.$$ Then, by Gordan's theorem of alternative, $$\sum_{j=0}^m \bar{\lambda_j}
~{}^C_c\nabla^{\alpha}_x f_{j,_{\alpha,\beta}}(x)=0,~~\bar{\lambda}\in\mathbb{R}^m_+\setminus\{0^m\}$$ has no solution. This contradicts (\ref{gd1}). Therefore, $x$ is a critical point.  
\end{proof}
From the above lemma, we can conclude that if $x$ is a non-critical point, then $d(x)\neq 0$ and $t(x)<0$. Next, we use the following Armijo-type line search technique to find a suitable step length:
\begin{equation}
\quad f_j(x + \eta\, d(x)) \leq f_j(x) + \sigma\, \eta\, t(x) ~~~\forall~j \label{amj1}
\end{equation}
where $\sigma \in(0,1)$ is pre specified. In the next lemma we justify that above line search technique is well defined.
\begin{lemma}
    Suppose $(t(x),d(x))$ is the optimal solution of $P_1(x)$ for some non critical point $x$. Then inequality (\ref{amj1}) holds for all $\eta>0$ sufficient small.
\end{lemma}
\begin{proof}
    Since $x$ is a non-critical point $d(x)\neq 0$  and $t(x)<-\frac{1}{2} \|d(x)\|^2$. Then for any $\eta>0$,
\begin{eqnarray*}
    f_j(x+\eta d(x)) &=& f_j(x) + \eta~{}^C_c\nabla^{\alpha}_x f_{j,_{\alpha,\beta}} (x)^Td(x) +  \frac{\eta^2}{2} C_{2,\alpha,\beta}  d(x)^T H_j(x) d(x) \\
    &\leq & f_j(x)+\eta t(x)+ \frac{\eta^2}{2} C_{2,\alpha,\beta}  d(x)^T H_j(x) d(x)\\
    &=& f_j(x)+\sigma \eta t(x) + (1-\sigma)\eta t(x)+  \frac{\eta^2}{2} C_{2,\alpha,\beta}  d(x)^T H_j(x) d(x)\\
    &\leq&  f_j(x)+\sigma \eta t(x)-\frac{(1-\sigma)}{2} \eta \|d(x)\|^2+ \frac{\eta^2}{2} C_{2,\alpha,\beta}~em_j(x) \|d(x)\|^2
\end{eqnarray*}
where $em_j(x)$ is the maximum eigen value of $H_j(x)$. Since $\sigma<1$, for $\eta>0$ sufficiently small we have $\frac{\eta^2}{2} C_{2,\alpha,\beta}~em_j(x) \|d(x)\|^2-\frac{(1-\sigma)}{2} \eta \|d(x)\|^2<0.$ Hence inequality (\ref{amj1}) holds for every $\eta>0$ sufficiently small.     
\end{proof}
\subsection{Algorithm}
In this subsection, an algorithm for $(MOP)$ is proposed based on theoretical aspects developed so far. The algorithm is initiated with some initial approximation $x^0$. At any iterating point $x^k$, the subproblem $SP_1(x^k)$ is solved to find a suitable descent direction. Next, for any non critical $x^k$, a suitable step length $\eta_k$ satisfying (\ref{amj1}) is selected. The process is repeated until we find some approximate critical point. For simplicity rest of the paper $t(x^k)$ and $d(x^k)$ are denoted by $t^k$ and $d^k$ respectively.
\begin{Algorithm}\label{algorithm 1} [Multi-objective adaptive order Caputo fractional gradient descent method]
\begin{enumerate}
    \item {\it Initialization:} Input the objective function $f$, approximation point $x^0 \in \mathbb{R}^n$, a scalar $\sigma \in (0, 1)$, $r \in (0,1)$, and tolerance $\epsilon > 0$ and set  $k = 0$.
    \item {\it Solve the subproblem: }\label{ste2} Solve the subproblem $SP_1(x^k)$ to find $(t^k,d^k).$
    \item {\it Criticality check:} If $\|d^k\| < \epsilon$, then stop and declare $x^k$ as an approximate critical point. Otherwise, go to Step \ref{line_search}.
    \item {\it Line search:} Select $\eta_k$ as the first element in the sequence $\{1, r, r^2, \ldots\}$ that satisfies  $$
        f_j(x^k + \eta_k d^k) \leq f_j(x^k) + \sigma \eta_k t(x^k), \quad \forall j\in\{1,2,...,m\}.
    $$
\label{line_search}
    \item {\it Update:} Compute $x^{k+1} = x^k + \eta_k d^k$. Update $k \gets k + 1$ and go to Step \ref{ste2}.

\end{enumerate}
\end{Algorithm}

\section{Convergence analysis for quadratic Tikhonov regularization}\label{sec4}
In this section, we justify the convergence of algorithm  \ref{algorithm 1} for a quadratic multi-objective optimization problem where each objective function can be written as
\begin{equation}
\label{eq:3.1}
\min_{x} f_j(x) = \frac{1}{2} x^{\top} A_j x + {b_j}^{\top} x + c_j,
\end{equation}
where $x, {b} \in \mathbb{R}^{d}$, $A_j = (a_{ij}) \in \mathbb{R}^{d \times d}$ and $c_j \in \mathbb{R}$. Here, all $A_j$ are assumed to be positive symmetric definite of (\ref{eq:3.1}). The least squares problem is formulated as follows
\begin{equation}\label{eq3.9}
\min_{x} f_j(x) = \Big\{\frac{1}{2} \Big\| W_1^{\top} x -  {y_1} \Big\|^2, \frac{1}{2} \Big\| W_2^{\top} x -  {y_2} \Big\|^2, \dots, \frac{1}{2} \Big\| W_m^{\top} x -  {y_m} \Big\|^2\Big\}.
\end{equation}
Tikhonov regularization is used to reduce the effect of disturbances in the least-squares method. A small penalty term is added to the objective function to stabilize the solution and reduce sensitivity to noise. Tikhonov’s regularization is formulated as follows based on equation~(\ref{eq3.9}),

\begin{align}\label{equ l-s}
&\min \Big\{ \|W_1 y_1 - y_1^{*}\|^2 + \gamma_1 \|R_1^T(x-c)\|^2 ,\|W_2 y_2 - y_2^{*}\|^2 + \gamma_2 \|R_2^T(x-c)\|^2 ,\notag\\
&\dots, \|W_m y_m - y_m^{*}\|^2 + \gamma_m \|R_m^T(x-c)\|^2 \Big\}.
\end{align}

 solution from (\ref{equ l-s}) is,
\begin{equation}
x^*_{\text{Tik}} = \bar{x} + \Big( \sum_{j=1}^{m} W_j W_j^T + \gamma_j~R_j R_j^T \Big)^{-1} \Big( \sum_{j=1}^{m} W_j y_j + \sum_{j=1}^{m} W_j^T x \Big),
\end{equation}

where $\bar{x}$ is a point at the path of $x$, \quad $R_j = \Big(\sqrt{a_{11}},\dots, \sqrt{a_{nn}}\Big)^T_j$\\

Theorem \ref{thm5} proves that non-adaptive $MOAOCGD$ converges linearly to a Tikhonov regularized solution which is not a critical point of integer order gradient. Theorem~\ref{thm6} establishes an error bound for MOAOCFGD with respect to the optimal solution $x^*$.  
\begin{theorem}\label{thm5}
Let \( f_j \) denote the objective functions defined in equation~(\ref{eq:3.1}), with \( \alpha \in (0,1) \) and \( \beta \in \mathbb{R} \). Define the matrix
\[
(A_{\alpha,\beta})_j := \sum_{j=1}^{m} \Big[ \lambda_j (W_j W_j^\top) + \gamma_\alpha \, diag(\tilde{R}_j) \Big],
\]
and let \( I \) represent the identity matrix and \( (A_{\alpha,\beta})_j \) is positive definite, and let \( \sigma_{\max} \) denote its largest singular value. Set
\(
\alpha_k = \alpha, \quad \beta_k = \beta, \quad c_k = c, \quad \alpha_L = \alpha, \quad \beta_L = \beta, \quad c_L = c,
\)
and define the step size as
\[
\eta_k = \frac{\eta}{\sigma_{\max}}, \quad \text{where } 0 < \eta < 2.
\]
Then the $k$-th iterated solution satisfies
\[
\|x^k - x^{*}_{\text{Tik}}(\gamma)\|^2 \leq \|x^{0} - x^{*}_{\text{Tik}}(\gamma)\|^2 (1 + \frac{\eta} {\kappa_{\alpha,\beta}})^k,
\]
where $\kappa_{\alpha,\beta}$ is the condition number of $A_{\alpha,\beta}$, and $x^{*}_{\text{Tik}}$ is the solution of the Tikhonov regularization with
\[
\gamma_{\alpha} = \frac{1-\alpha}{2-\alpha},~~\gamma_{\alpha,\beta} = \beta - \frac{1 - \alpha}{2 - \alpha}, \quad \text{and} \quad \tilde{R}=  \Big( \sqrt{\sum_{p=1}^{m} w_{piq}^2} \Big).
\]
\end{theorem}

\begin{proof}
Let $c = (c_j)$ and $x = (x_j)$ and 
\begin{equation*}
M = diag\Big([^C_{c_1} D_{x_1}^{\alpha} x_1 \quad \cdots \quad ^C_{c_d} D_{x_d}^{\alpha} x_d ]\Big) \in \mathbb{R}^{d \times d}.
\end{equation*}

For $0 < \alpha < 1$, it can be checked that\cite{50}
\[
\Big( {}_{ {c}}^C \nabla_{x}^{\alpha} f_j(x) \Big)_j = {}^C_{c} D_{x}^{\alpha} x_j \sum_{j=1}^{m}\Big[ w_{ij}^2 \gamma_{\alpha}(x - c) + w_{ij}(  {w}_i^{\top} x - y_i ) \Big],
\]

where the equality uses lemma (see lemma A.1~\cite{50}) and $\gamma_{\alpha} = \frac{1 - \alpha}{2 - \alpha}$. Then, the Caputo fractional gradient of $f_j(x)$ is given by
\[
{}^{C}_{c}\nabla^{\alpha}_{x} f_j(x) = M \Big( \sum_{j=1}^{m} \Big[W_j W_j^\top x - W_j  {y}_j +  \gamma_{\alpha} \,  diag(\tilde{R}_j)(x - c)\Big] \Big),
\]

similarly,
\begin{align*}
\sum_{j=1}^{m}  \, {}^{C}_{c}\nabla^{\alpha+1}_{x} f_j(x) = M \sum_{j=1}^{m}  \Big[ diag(\mathrm{sign}(x - c)) \tilde{R}_j \Big]
\end{align*}
therefore,
\begin{align*}
\sum_{j=1}^{m} \lambda_j \, {}^{C}_{c}\nabla^{\alpha}_{x} f_j(x) + \beta \, diag(|x-c|) \, \sum_{j=1}^{m} \lambda_j \, {}^{C}_{c}\nabla^{\alpha+1}_{x} f_j(x) =
\end{align*}

\begin{align*}
&= M \sum_{j=1}^{m} \lambda_j \Big[ W_j W_j^\top x - W_j  {y}_j + \gamma_{\alpha} \, diag(\tilde{R}_j)(x - c) \Big] \\
&\quad + \beta \, diag(|x - c|) \, M \sum_{j=1}^{m} \lambda_j \Big[ diag(\mathrm{sign}(x - c)) \tilde{R}_j \Big] \\
&= M \sum_{j=1}^{m} \lambda_j \Big[ (W_j W_j^\top x - W_j  {y}_j) + (\beta + \gamma_\alpha) \, diag(\tilde{R}_j)(x - c) \Big]
\end{align*}

\begin{equation}
\label{eq:dk}
d^k = -\sum_{j=1}^{m} \lambda_j \Big[ (W_j W_j^\top x - W_j  {y}_j) + \gamma_{\alpha ,\beta} \,  diag(\tilde{R}_j)(x - c) \Big]
\end{equation}

\noindent where $\gamma_{\alpha, \beta} = \gamma_\alpha + \beta$, and $ {y}_j = W_j^\top x^*$.

\[
d^k = -\sum_{j=1}^{m} \lambda_j \Big[W_j W_j^T x -\gamma_{\alpha, \beta}  diag(\tilde{R}_j) c + \gamma_{\alpha, \beta}  diag (\tilde{R_j})x- W_jW_j^T x^{*}\Big]
\]

\noindent where $k_j = (W_j W_j^\top)^{-1} \,  diag(\tilde{R}_j)$. From there,
\[
d^k = -\sum_{j=1}^{m} \lambda_j W_j W_j^T \Big[x \Big(I+k_j \gamma_{\alpha, \beta}\Big)x -\Big(x^* + \gamma_{\alpha, \beta} k_j  c \Big)\Big]
.\]
\begin{align*}
\text{Hence,}~x^*_{\alpha, \beta} = \Big( \sum_{j=1}^{m} \lambda_j (W_j W_j^\top) (I + k_j \gamma_{\alpha ,\beta}) \Big)^{-1} 
\Big( \sum_{j=1}^{m} \lambda_j (W_j W_j^\top) \Big( x^* + \gamma_{\alpha ,\beta} k_j c \Big) \Big).
\end{align*}

So, from (\ref{eq:dk}), $d^k$ can be written as
\begin{align*}
d^k& = -\sum_{j=1}^{m} \lambda_j \Big[ W_j W_j^\top x - W_j  {y}_j + \gamma_\alpha \,  diag(\tilde{R}_j) x - \gamma_\alpha \,  diag(\tilde{R}_j) x_{\alpha,\beta}^* \Big]\\
&= - \sum_{j=1}^{m} \lambda_j \Big[ \Big( W_j W_j^\top + \gamma_\alpha \, diag(\tilde{R}_j) \Big)(x - x_{\alpha,\beta}^{*}) \Big].
\end{align*}

\begin{flushleft}

\begin{align*}
\text{Now,}~x^{k+1} - x^*_{\alpha, \beta} 
&= x^k - \eta_k \, d^k - x^*_{\alpha, \beta} \\
&= x^k + \eta_k \sum_{j=1}^{m} \lambda_j \Big[ W_j W_j^\top x + \gamma_\alpha \, diag(\tilde{R}_j)(x-c_{\alpha, \beta}) \Big] - x^*_{\alpha, \beta} \\
&= \Big( x^k - x^{*}_{\alpha,\beta} \Big) \Big( I + \eta_k \sum_{j=1}^{m} \lambda_j \Big[ W_j W_j^\top + \gamma_\alpha \, diag(\tilde{R}_j) \Big] \Big)
\end{align*}.
\end{flushleft}

Let $~(A_{\alpha,\beta})_j = \sum_{j=1}^{m} \Big[ \lambda_j (W_j W_j^\top) + \gamma_\alpha \, diag(\tilde{R}_j) \Big]$ be a matrix and $I$ be the identity matrix. Let $\sigma_{\max}$ be the largest singular value of $(A_{\alpha,\beta})_j$. Let $k$ be the condition number of $A_{\alpha,\beta}$. Suppose $\eta_k = \frac{\eta}{\sigma_{\max}}, \eta \in (0,2)$. Thus, the equation can be written as

\[
\|x^{k+1} - x_{\alpha,\beta}^{*} \|^2 \leq \|x^{0} - x_{\alpha,\beta}^{*} \|^2 \Big(1 + \frac{\eta}{k_{\alpha,\beta}}\Big)^k.
\]

\begin{align*}
&\text{Now},~x_{\alpha,\beta}^{(k)} = \Big( \sum_{j=1}^{m} \lambda_j (W_j W_j^\top)(I + k_j \gamma_{\alpha,\beta}) \Big)^{-1} \Big( \sum_{j=1}^{m} \lambda_j (W_j W_j^\top)(x^* + \gamma_{\alpha,\beta} k_j c) \Big)\\
&= \Big( \sum_{j=1}^{m} \lambda_j (W_j W_j^\top)(I + k_j \gamma_{\alpha,\beta}) \Big)^{-1} \Big( \sum_{j=1}^{m} \lambda_j (W_j W_j^\top)\Big[(I + \gamma_{\alpha,\beta} k_j) c +(x^* - c)\Big] \Big)\\
&= \Big( \sum_{j=1}^{m} \lambda_j (W_j W_j^\top)(I + k_j \gamma_{\alpha,\beta}) \Big)^{-1}\\
&
\Big( \sum_{j=1}^{m} \lambda_j (W_j W_j^\top)(I + \gamma_{\alpha,\beta} k_j) c + \sum_{j=1}^{m} \lambda_j (W_j W_j^\top)(x^* - c) \Big)\\
&= c + \Big( \sum_{j=1}^{m} \lambda_j \Big[W_j W_j^\top +  diag(\tilde{R}_j) \gamma_{\alpha,\beta}\Big] \Big)^{-1} \Big(\sum_{j=1}^{m} \lambda_j (W_j W_j^\top)(x^* - c)\Big)
= x^{k}_{\text{Tik}}.
\end{align*} 
\end{proof}
Remarks: A higher value of ${k_{\alpha,\beta}}$ denotes slow convergence. 

\begin{theorem}\label{thm6}
Let $\{\gamma_s\}_{s \geq 1}$ be a non-negative convergent sequence converges to $0$. Let $\{k_s\}_{s \geq 1}$ be a sequence of positive integers. It is assumed that $k$ is  condition number of $(A_{\alpha,\beta})_j$. Let $\eta_{s,k}= \eta_s$ be the step-length of the $k$-th iteration at the $s$-th stage for some $\eta_0 > 0$. Let
\(
R_s = \Big| 1 + \frac{\eta_s}{k_s} \Big|^\frac{k_s}{2} \text{ for all } s.
\) Then, solution of the adaptive order Caputo gradient fractional gradient descent of $(MOP)$ after $s$ stages satisfies
\begin{align*}
\Big\| x_s^{k} - x^* \Big\|
&\leq \prod_{l=1}^{s-1} R_{s-l} \Big\| x^{0} - x^*_{\text{Tik}}(\gamma_1) \Big\| \\
&\quad + C \Big\{ \sum_{i=1}^{s-1} \Big( \prod_{l=0}^{i-1} R_{s-l} \Big) \Big\| \gamma_{s-l} - \gamma_{s-i+1} \Big\| + \|\gamma_s\| \Big\}
\end{align*}

where $x^*_{\text{Tik}}$ is the solution of Tikhonov regularization solution, and $C$ is a constant which depends only on $W_j$, $R_s$, $x^*$, $c$, and the range of the $\{\gamma_s\}$ for $j\in\{1,2,\dots,m\}.$

\end{theorem}
\begin{proof}
Let $ x_{\gamma}^* $ be the solution to the Tikhonov regularization problem, where
$
\gamma = \beta - \frac{1 - \alpha}{2 - \alpha}.
$ Here, $ \gamma_s = \beta_s - \frac{1 - \alpha_s}{2 - \alpha_s} $, and $ \gamma_s \in \Big[0, \beta_s - \frac{1}{2}\Big) $ for all $ s $. It follows that
$
\lim_{s \rightarrow0} \gamma_s = \lim_{s \rightarrow0} \Big( \beta_s - \frac{1 - \alpha_s}{2 - \alpha_s} \Big) = 0,
$
where $ \beta_s \in \Big[ \frac{1 - \alpha_s}{2 - \alpha_s}, \infty \Big) $ and $ \alpha_s \in (0,1] $, so that the stated condition holds. Let $ (A_{\alpha, \beta})_j $ be a matrix defined earlier, and let $ k_{\alpha, \beta} $ denote the associated condition number. For each stage $ s = 1, 2, \dots $, we have
\vspace{-1 mm}
\begin{align*}
&\Big\| x_s^{k_s} - x^*_{\gamma_s} \Big\|^2 
\leq r_s \Big\| x^{0} - x^*_{\gamma_s} \Big\|^2, 
\quad \text{where} \quad
\Big(1 + \frac{\eta}{k_{\alpha_s,\beta_s}} \Big)^{k_s} = r_s^{(k_s)}\\
&\text{that implies}~~\Big\| x_s^{k_s} - x^*_{\gamma_s} \Big\| 
\leq r_s^{\frac{k_s}{2}} \Big\| x^{0} - x^*_{\gamma_s} \Big\|.
\end{align*}

Define, $ \epsilon_s = \Big\| x^{0} - x^*_{\gamma_s} \Big\| $, $ e_s = \Big\| x^*_{\gamma_s} - x^*_{\gamma_{s+1}} \Big\| $, and $ R_s = r_s^{\frac{k_s}{2}} $. Note that $ x_s^{(k_s - 1)} = x_s^{0} $ for all $ s $. Therefore, we have,
$\epsilon_s \leq R_{s-1} \epsilon_{s-1} + e_{s-1}.$
At stage $ s - 1 $, the following inequality holds,
\begin{align*}
&\Big\| x_s^{(k_{s-1})} - x^*_{\gamma_s} \Big\| \leq R_{s-1} \Big\| x_{s-1}^{0} - x^*_{\gamma_{s-1}} \Big\|\\
&\text{implies}~\Big\| x_s^{0} - x^*_{\gamma_s} \Big\| = \epsilon_s \leq R_{s-1} \epsilon_{s-1} + e_{s-1},
\end{align*}
where the error bound $e_s $ is given by $ \Big\| x^*_{\gamma_s} - x^*_{\gamma_{s+1}} \Big\|.$
Applying this relationship recursively, we obtain,
$
R_s \epsilon_s \leq \sum_{k=1}^{s} \Big( \prod_{l=0}^{k-1} R_{s-l} \Big) e_{s-k},$ where $e_0 = \epsilon_1.
$ Thus, the final error bound becomes,
$$
\Big\| x_s^{k_s} - x_0^* \Big\| \leq \Big( \sum_{k=1}^{s} \Big( \prod_{l=0}^{k-1} R_{s-l} \Big) e_{s-k} \Big) \Big\| x_{\gamma_s}^* - x^* \Big\|.$$\\
For any $ \gamma, \gamma_j \in \Big[0, \beta - \tfrac{1}{2} \Big] $, we have
\begin{align*}
x^*_{\gamma'} - x^*_{\gamma} 
&= \bar{x} + \Big( \sum_{j=1}^{m} W_j W_j^\top + \gamma_j R_j R_j^\top \Big)^{-1} 
\Big( \sum_{j=1}^{m} W_j y_j - W_j W_j^\top \bar{x} \Big) - \bar{x} \\
&\quad - \Big( \sum_{j=1}^{m} W_j W_j^\top + \gamma'_j R_j R_j^\top \Big)^{-1} 
\Big( \sum_{j=1}^{m} W_j y_j - W_j W_j^\top \bar{x} \Big) \\
&= \Big[ \Big( \sum_{j=1}^{m} W_j W_j^\top + \gamma_j R_j R_j^\top \Big)^{-1} 
- \Big( \sum_{j=1}^{m} W_j W_j^\top + \gamma'_j R_j R_j^\top \Big)^{-1} \Big] \\
&\quad \cdot \Big( \sum_{j=1}^{m} W_j y_j - W_j W_j^\top \bar{x} \Big) \\
&= \Big( \sum_{j=1}^m W_j W_j^\top + \gamma_j R_j R_j^\top \Big)^{-1} \\
&\Big[ I - \Big( \sum_{j=1}^{m} W_j W_j^\top + \gamma'_j R_j R_j^\top \Big)^{-1} 
\Big( \sum_{j=1}^{m} W_j W_j^\top + \gamma_j R_j R_j^\top \Big) \Big]~ \\
&\cdot \Big( \sum_{j=1}^{m} W_j y_j - W_j W_j^\top \bar{x} \Big) \\
&= \Big( \sum_{j=1}^{m} W_j W_j^\top + \gamma_j R_j R_j^\top \Big)^{-1} 
\Big[ I - \frac{ \sum_{j=1}^m W_j W_j^\top + \gamma'_j R_j R_j^\top }
{ \sum_{j=1}^m W_j W_j^\top + \gamma_j R_j R_j^\top } \Big] \\
&\quad \cdot \Big( \sum_{j=1}^{m} W_j y_j - W_j W_j^\top \bar{x} \Big) \\
&= (\gamma_j - \gamma'_j) \sum_{j=1}^{m} R_j R_j^\top 
\Big( \sum_{j=1}^{m} W_j W_j^\top + \gamma R_j R_j^\top \Big)^{-1} 
\Big( \sum_{j=1}^{m} W_j y_j - W_j W_j^\top \bar{x} \Big).
\end{align*}
Assume, $ B_{\max} = \sup\limits_{\gamma_j \in \Big[0, \beta - \tfrac{1}{2} \Big]} 
\Big\| \Big( \sum_{j=1}^{m} W_j W_j^\top + \gamma R_j R_j^\top \Big)^{-1} \Big\| $ is finite, and let 
\[
C = B^2 \Big\| \sum W_j W_j^\top \Big\| \Big\| \sum R_j \Big\|^2 \Big\| x^* - c \Big\|.
\]

Then,
\[
\Big\| x^*_{\gamma_s} - x_0^* \Big\| \leq C |\gamma_s|, \quad 
e_s = \Big\| x^*_{\gamma_s} - x^*_{\gamma_{s+1}} \Big\| \leq C \Big| \gamma_s - \gamma_{s+1} \Big|, \quad \forall s \geq 1.
\]

\begin{align*}
\text{Now,}~\Big\| x_s^{k_s} - x^* \Big\| 
&= \Big\| x_s^{k_s} - x^*_0 \Big\| \quad \text{(as } x_0^* = x^* \text{)} \\
&\leq \sum_{k=1}^{s} \Big( \prod_{l=0}^{k-1} R_{s-l} \Big) e_{s-k} + \Big\| x^*_{\gamma_s} - x_0^* \Big\| \\
&\leq \sum_{k=1}^{s-1} \Big( \prod_{l=0}^{k-1} R_{s-1} \Big) e_{s-k} + \Big\| x^*_{\gamma_s} - x_0^* \Big\| + e_0 \Big( \prod_{l=0}^{k-1} R_{s-l} \Big) \\
&\leq \Big( \prod_{l=0}^{s-1} R_{s-l} \Big) \Big\| x^{0} - x_{\gamma_1}^* \Big\| + 
C \Big\{ \Big( \sum_{k=1}^{s-1} \prod_{l=0}^{k-1} R_{s-l} \Big) \Big| \gamma_{s-k+1} - \gamma_{s-k} \Big| + |\gamma_s| \Big\}.
\end{align*}

Let, $ k_s = k $, and suppose $ R_s < \rho^k $, where $ \rho \in (0,1) $. Then we have
\[
\lim_{k \to \infty} \Big\| x_s^{k} - x^* \Big\| \leq C |\gamma_s|, \quad 
\text{and} \quad 
\lim_{s \to \infty} \Big\| x_s^{k} - x^* \Big\| < \frac{C \rho^k}{1 - \rho^k}.
\]

Therefore, 
\(
\lim_{k,s \to \infty} \Big\| x_s^{k} - x^* \Big\| = 0.
\), this completes the proof.
\end{proof}

\section{Numerical example}\label{sec5}
In this section, we present a series of numerical experiments that both validate our theoretical results and evaluate the efficacy of the proposed the Caputo fractional gradient descent algorithm in multi‐objective optimization tasks. In the first experiment, we benchmark its convergence behavior against the classical gradient descent method(GD). The second example demonstrates its applicability and performance in training a neural network model under a multi‐objective loss formulation.  A numerical example helps to understand this methodology better. A Python-based code is developed for the proposed method. The subproblem is solved using Python solver {\it cvxopt}. The stopping criteria of the method is $\|d^k\|<10^{-4}$ or maximum $2000$ iterations. To generate approximate Pareto fronts, a set of $100$ uniformly distributed random points between initial points $lb$ and final points $ub$ is considered and the proposed algorithm is executed individually. 
\begin{e1}
{\it Quadratic multi-objective optimization problem with random data:} \emph{Consider the optimization problem,
\begin{equation}\label{eq:quadratic_mop}
\min_{x \in \mathbb{R}^n} 
\Bigl(\tfrac{1}{2} x^\top A_1 x + b_1^\top x,\;
      \tfrac{1}{2} x^\top A_2 x + b_2^\top x\Bigr).
\end{equation}
 where $A_j= W_j W_j^\top$, $b_j = -A_jy_j$ for $j=1,2$ and  \( W_j \in \mathbb{R}^{m \times n} \) is generated by sampling each entry independently from the uniform distribution over $(-1,1)$.}
 \end{e1}
Fix the Tikhonov regularization parameters 
\( \gamma \in \{0.15, 0.25, 0.5, 0.75, 1, 10\} \), 
and use a memory length \( L = 1 \), with a terminal vector \( c = 0 \). As initial guesses, we generate a grid of \(100\) points by uniformly subdividing the segment from \((1.01, \dots, 1.01)\) to \((10, \dots, 10)\). All experiments are conducted with \( m = n = 100 \). The $\ell_2$ distance to the regularized solution \( x_{\mathrm{Tik}}(\gamma) \) is reported with respect to the number of iterations for each value of \( \gamma \in \{0.15, 0.25, 0.5, 0.75, 1, 10\} \).

Table~\ref{tab:gd_vs_moaocfgd} presents a comparative analysis of classical gradient descent and the proposed MOAOCFGD method for various regularization values. Two key metrics are highlighted: the condition number and CPU time. A high condition number typically indicates slower convergence, while longer CPU times reinforce this observation. At \( \gamma = 0.15 \), GD yields a high condition number of approximately \( 5.14 \times 10^1 \) and requires \( 0.589 \) seconds(s) of CPU time. In contrast, MOAOCFGD significantly reduces the condition number to \( 2.59 \) and completes in just \( 0.10 \) seconds. As the value of \( \gamma \) increases, both methods demonstrate improved numerical conditioning. However, MOAOCFGD consistently achieves better performance in all values. For example, at \( \gamma = 1.00 \), the condition number for GD drops to \( 3.95 \) with a CPU time of \( 0.104 \) seconds, while MOAOCFGD achieves a lower condition number of \( 1.90 \) in only \( 0.07 \) seconds. Even at the high regularization value \( \gamma = 10.0 \), GD exhibits a condition number of \( 18.8 \) with \( 0.148 \) seconds of CPU time. MOAOCFGD produces a comparable condition number of \( 19.9 \), but in a reduced time of \( 0.06 \) seconds. These results clearly indicate that MOAOCFGD consistently yields smaller condition numbers and faster runtimes, demonstrating a more efficient convergence behavior than standard GD. Furthermore, Figure~\ref{fig:comparison-pareto1}, \ref{fig:comparison-pareto2} and \ref{fig:comparison-pareto3}, Pareto fronts are plotted corresponding to \( \gamma = \{0.15, 0.25, 0.50, 0.75, 1 ,10\} \), which illustrates the trade-offs and the optimization landscape more effectively.

\begin{table}
    \centering
    \begin{tabular}{%
        >{\centering\arraybackslash}p{1.2cm}
        >{\centering\arraybackslash}p{1.5cm}
        >{\centering\arraybackslash}p{1.5cm}
        >{\centering\arraybackslash}p{1.5cm}
        >{\centering\arraybackslash}p{1.5cm}
    }
        \toprule
        & \multicolumn{2}{c}{Gradient Descent (GD)}
        & \multicolumn{2}{c}{MOAOCFGD} \\
        \cmidrule(lr){2-3}\cmidrule(lr){4-5}
        Value of $\gamma$
        & Condition no.\ & CPU time (s)
        & Condition no.\ & CPU time (s) \\
        \midrule
        0.15 & 5.14e+01 & 0.589 & 2.59e+00 & 0.10 \\
        0.25 & 6.93e+00 & 0.400 & 3.22e+00 & 0.08 \\
        0.50 & 9.33e+00 & 0.218 & 6.68e+00 & 0.13 \\
        0.75 & 5.92e+00 & 0.078 & 5.72e+00 & 0.12 \\
        1.00 & 3.95e+00 & 0.104 & 1.90e+00 & 0.07 \\
        10.0 & 1.88e+01 & 0.148 & 1.99e+01 & 0.06 \\
        \bottomrule
    \end{tabular}
    \caption{Comparison between MOGD and MOAOCFGD across Tikhonov parameter \(\gamma\).}
    \label{tab:gd_vs_moaocfgd}
\end{table}

\begin{figure}
  \centering
  \begin{subfigure}[b]{0.45\linewidth}
    \centering
    \includegraphics[width=\linewidth,height=2.75cm]{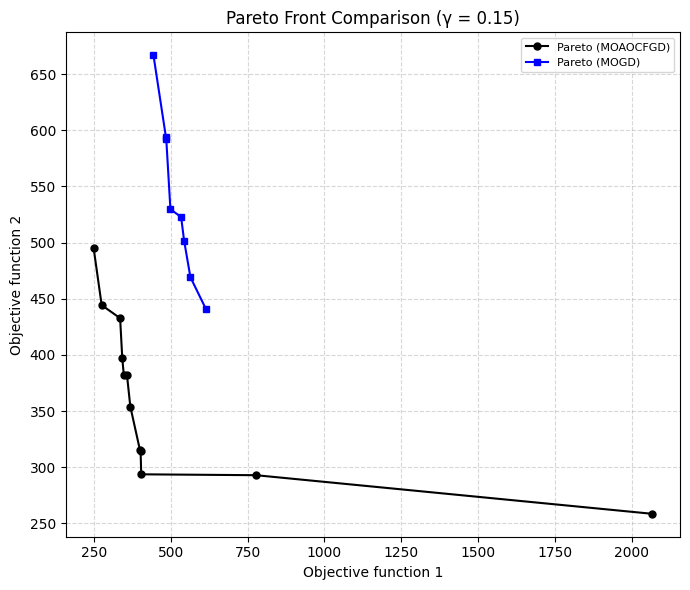}
   
  \end{subfigure}
  \hfill
  \begin{subfigure}[b]{.45\linewidth}
    \centering
    \includegraphics[width=\linewidth,height=2.75cm]{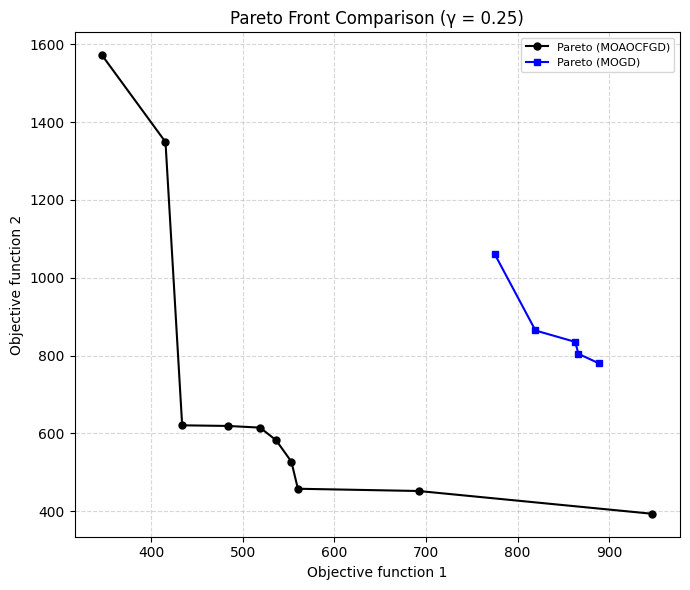}
    \end{subfigure}
  \caption{MOAOCFGD versus MOGD at $\gamma$ values of $0.15$ and $0.25$.}
  \label{fig:comparison-pareto3}
  
\end{figure}
\begin{figure}
  \begin{subfigure}[b]{0.45\linewidth}
    \centering
    \includegraphics[width=\linewidth,height=2.75cm]{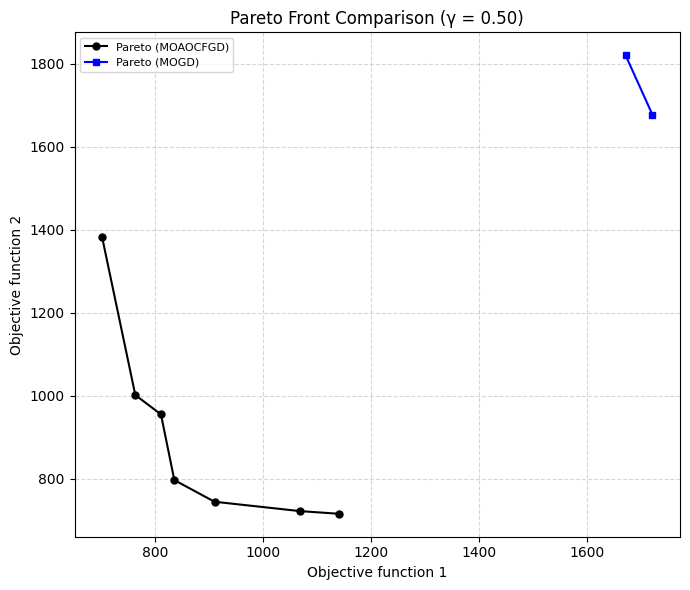}
    
  \end{subfigure}
  \hfill
  \begin{subfigure}[b]{0.45\linewidth}
    \centering
    \includegraphics[width=\linewidth,height=2.75cm]{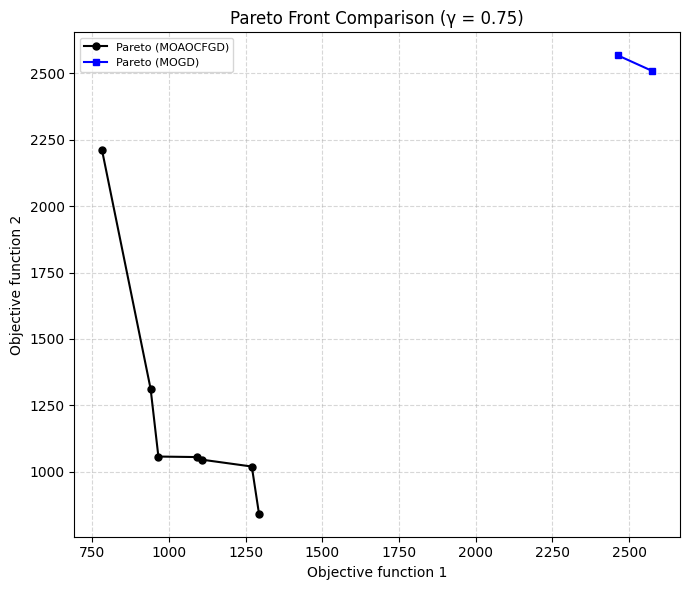}
    \end{subfigure}
  \caption{MOAOCFGD versus MOGD at $\gamma$ values of $0.5$ and $0.75$.}
  \label{fig:comparison-pareto2}

  \end{figure}
  \begin{figure}
  \begin{subfigure}[b]{0.45\linewidth}
    \centering
    \includegraphics[width=\linewidth,height=2.75cm]{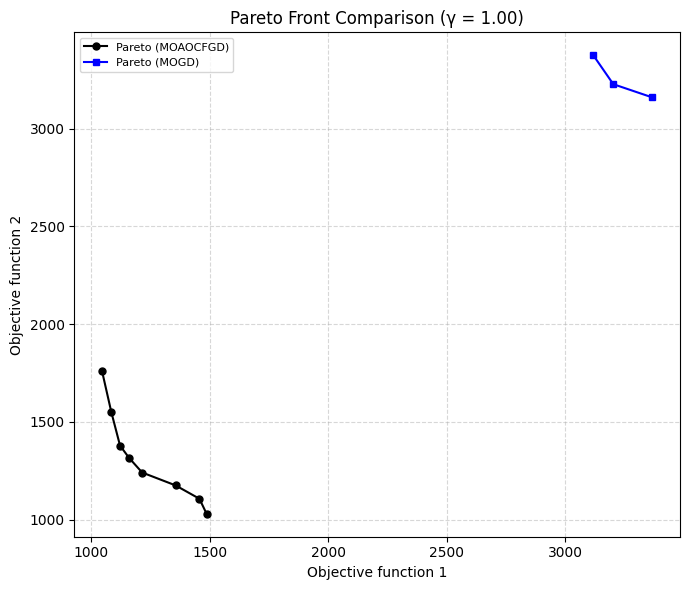}
    
  \end{subfigure}
  \hfill
  \begin{subfigure}[b]{0.45\linewidth}
    \centering
    \includegraphics[width=\linewidth,height=2.75cm]{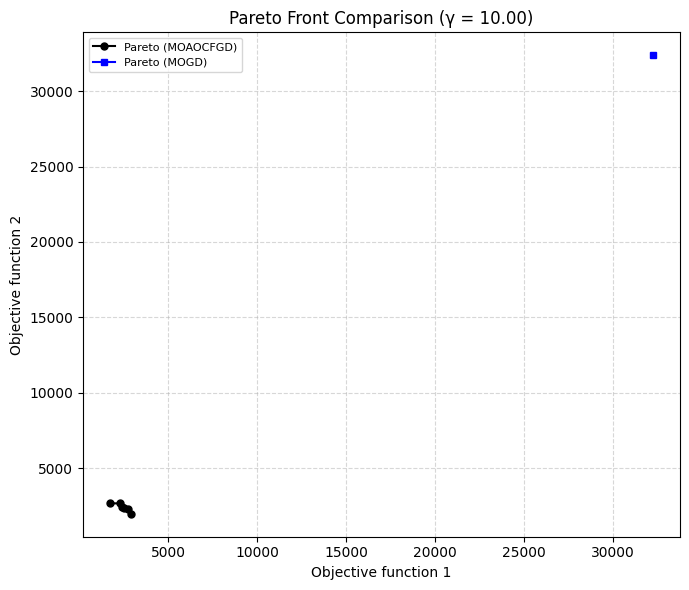}
   \end{subfigure}
  \caption{MOAOCFGD versus MOGD at $\gamma$ values of $1$ and $10$.}
  \label{fig:comparison-pareto1}
\end{figure}

\begin{e1}{Implementation in neural networks:} \emph{This exapmle illustrates the methodology in the application field of numerical examples. The methodology is also applicable to non-smooth functions. A numerical example helps to understand this methodology better. A Python-based code is developed for the proposed methodology. For this following steps are follows.}
\end{e1}
\begin{enumerate}
    \item The design search exploration(DSE) algorithm was evaluated on ANN applications focusing on the standard MNIST and CIFAR-10 image recognition datasets, where all ANN models were rigorously trained and tested.

    \item The activation function used is ReLU (Rectified Linear Unit),\\ mathematically defined as 
    \[
    \text{ReLU}(x) = \max(0, x).
    \]
    \item The output layers are comprised of softmax units, with the softmax function formulated as 
    \[
    \text {softmax}(z_i) = \frac{e^{z_i}}{\sum_{j=1}^{K} e^{z_j}},
    \]
    where $z_i$ denotes the score of class $i$ and $K$ is the total number of classes.
    \item Validation of the proposed algorithm was undertaken on two distinct image recognition databases: The MNIST database, containing $28 \times 28$ pixel grayscale images, and the CIFAR-10 dataset, which includes $32 \times 32$ pixel images in RGB colour. The learning rates were set between 0.1 and 0.8, with the training duration extending from 100 to 200 epochs.
    \item $||d^k||< 10 ^{-4}$ or maximum $200$ iteration is considered as stopping criteria. To find out execution time stopping criteria is not use that time.
\end{enumerate}
This structured approach facilitates a comprehensive evaluation of the algorithm's performance across diverse settings. The network settings used in the experiment are summarized below. Certain configurations were modified to explore the effects of varying network designs.
However, kept some settings the same for all experiments
\begin{enumerate}
    \item \text{Output Layer:} A softmax layer is used to compute the probability distribution over all classes.
    \item \text{Loss Function:} Categorical cross-entropy is employed to evaluate the accuracy of the predictions.
    \item \text{Batch Normalization:} Batch normalization is applied to stabilize and enhance the training process, with the exception of the reduced MNIST experiments, where it is omitted.
\end{enumerate}
Table~\ref{Table 1} presents the specific settings varied in experiments. Reference~\cite{18} introduces the DSE method, which employs response surface modelling (RSM) techniques to predict classification accuracy and to automate the design of artificial neural networks(ANN) hyperparameters. This method has been validated with multi-layer perception (MLP) and convolutional neural network (CNN) architectures targeting the CIFAR-10 and MNIST image recognition datasets~\cite{33,34}. Furthermore, the study in~\cite{18} develops a multi-objective hyper-parameter optimization approach that can effectively reduce the implementation costs of ANN.

\par To evaluate the effectiveness of the method in identifying optimal solutions, a comparison is conducted with an exhaustive search approach. The comparison involved designing MLP models for the MNIST dataset, limiting the search space to parameters listed in the first section of Table \ref{Table 1}, which contains $10,000$ potential solutions. Each solution was trained and evaluated. The performance of the MOAOCFGD algorithm over 200 iterations, where each iteration involves designing, training, evaluating, and updating the model, is depicted. The results demonstrate that the MOAOCFGD algorithm effectively represents the true Pareto optimal front, evaluating only a limited number of solutions. Notably, the outcomes from the MOAOCFGD algorithm generally surpassed those obtained from the exhaustive search in several cases. Most solutions assessed by the MOAOCFGD algorithm stayed near the true Pareto front, avoiding Pareto-dominated regions.

Table~\ref{Table 2} presents a comparative analysis between the proposed MOAOCFGD algorithm and the DSE approach detailed in \cite{18}. The results clearly indicate that MOAOCFGD outperforms the DSE method across all evaluated metrics. Specifically, MOAOCFGD achieves lower average distance from reference set(ADRS) (see\cite{18}) values and reduced iteration times, highlighting its superior accuracy and efficiency in approximating the Pareto front. These improvements are particularly notable during the initial 30 iterations, where the algorithm rapidly converges toward the true Pareto front, with subsequent iterations yielding more refined adjustments.

Figure~\ref{fig1} and \ref{fig2} illustrate the progression of the approximated Pareto optimal set generated by MOAOCFGD across different iteration counts. The first two images show the results for Restricted MNIST and MNIST (MLP) using both the DSE method and the proposed method. From the plots, it is evident that the proposed method demonstrates better convergence performance compared to the DSE approach. The second two images depict the Pareto fronts obtained using the DSE method for the MNIST (CNN) and CIFAR-10 (CNN) datasets, respectively. The subsequent two images illustrate the results produced by the proposed MOAOCFGD method. It is evident from the comparison that the proposed approach yields a more accurate  Pareto front, indicating improved performance over the baseline method. As the iterations increase, the algorithm solutions align closely with the true Pareto front, reflecting a steady improvement in approximation quality. The figure shows a clear reduction in the ADRS, with values dropping from $19.74\%$ at iteration $20$ to $4.78\%$ at iteration $100$, and reaching as low as $2.45\%$ by iteration $200$. This information is shown in Table \ref{Table 1} and Table \ref{Table 2}.  
\begin{table}[!http]
\begin{center}
\begin{tabular}{@{}cccccc@{}}
\toprule
\text{Target} & \text{Parameter} & \text{Range} & \text{Steps} & \text{Scale}\\ \hline
Restricted MNIST (MLP) & Number of FC layers & 1 to 3 & 3 & Linear \\
 & Nodes per FC layer & 10 to 200 & 10 & Log \\
 & Learning rate & 0.01 to 0.8 & 9 & Log \\
 & Activation function & ReLU and softmax & 1 & - \\
 & Training algorithm & PW & 1 & - \\
 & Batch sizes & 300 & 1 & - \\
 & Training epochs & 10 & 1 & - \\
\\\hline
MNIST (MLP) & Number of FC layers & 1 to 3 & 3 & Linear \\
 & Nodes per FC layer & 10 to 500 & 85 & Log \\
 & Learning rate & 0.001 to 0.8 & 16 & Log \\
 & Activation function & ReLU or tanh & 2 & - \\
 & Training algorithm & PW & 1 & - \\
 & Batch sizes & 200 & 1 & - \\
 & Training epochs & 10 & 1 & - \\
\\ \hline
MNIST (CNN) & Number of CNN layers & 1 to 3 & 3 & Linear \\
 & Filters per CNN layer & 12 to 128 & 12 & Log \\
 & Learning rate & 0.012 to 0.8 & 12 & Log \\
 & Filter kernel size & 1x1 to 5x5 & 3 & Linear \\
 & Max-pool size & 2x2 to 4x4 & 3 & Linear \\
 & Activation function & ReLU and softmax & 1,1 & - \\
 & Training algorithm & 0.973 & 1 & - \\
 & Batch sizes & 200 & 1 & - \\
 & Training epochs & 5 & 1 & - \\
\\ \hline
CIFAR-10 (CNN) & Number of CNN layers & 1 to 2 & 3 & Linear \\
 & Filters per CNN layer & 16 to 512 & 20 & Log\\
 & Learning rate & 0.012 to 0.8 & 1 & - \\
 & Filter kernel size & 3x3 to 9x9 & 4 & Linear \\
 & Max-pool size & 2x2 to 3x3 & 2 & Linear \\
 & Activation function & ReLU & 1 & - \\
 & Training algorithm & Adam  & 1 & - \\
 & Batch sizes & 3$\times$ 3 & 1 & - \\
 & Training epochs & 5 & 1 & - \\
\\ \hline

\end{tabular}
\end{center}
\caption{Experimental design-space parameters}
\label{Table 1}

\end{table}
\noindent

\begin{table}[!http]
\noindent
\begin{center}
\begin{tabular}{@{}cccccc@{}}
\hline
\multicolumn{1}{|c|} {\text{Type}}& \multicolumn{2}{|c|}{\text{DSE}} & \multicolumn{2}{|c|}{\text{MOAOCFGD}} \\
\hline
\text{Iterations} & \text{ADRS value} & \text{Exec. time(sec.)} & \text{ADRS value} & \text{Exec. time (sec.)} \\
\hline

\multicolumn{5}{|c|}{\text{Function Type: Restricted MNIST(MLP)}} \\
\hline
10  & 30\%  & 120   & 21.55\%  & 34.04 \\
20  & 21\%  & 241   & 19.74\%  & 82.29 \\
30  & 8.4\% & 364   & 7.32\%   & 115.31 \\
50  & 7.1\% & 672   & 7.29\%   & 178.32 \\
70  & 6.1\% & 984   & 4.95\%   & 245.22 \\
100 & 5.0\% & 1524  & 4.78\%   & 339.12 \\
150 & 4.5\% & 2712  & 2.51\%   & 513.04 \\
200 & 3.6\% & 4206  & 2.45\%   & 591.27 \\
\hline

\multicolumn{5}{|c|}{\text{Function Type: MNIST(MLP)}} \\
\hline
10  & 28\%  & 132   & 20.96\%  & 33.93 \\
20  & 20\%  & 270   & 22.45\%  & 75.70 \\
30  & 7.5\% & 408   & 6.75\%   & 109.74 \\
50  & 6.8\% & 720   & 7.02\%   & 174.91 \\
70  & 5.8\% & 1050  & 4.64\%   & 240.74 \\
100 & 4.9\% & 1620  & 4.49\%   & 322.46 \\
150 & 5.1\% & 2316  & 4.49\%   & 435.53 \\
200 & 4.8\% & 2538  & 4.64\%   & 584.13 \\
\hline

\multicolumn{5}{|c|}{\text{Function Type: MNIST(CNN)}} \\
\hline
10  & 32\%  & 180   & 27.39\%  & 53.92 \\
20  & 22\%  & 360   & 21.37\%  & 114.73 \\
30  & 9\%   & 540   & 10.93\%  & 170.73 \\
50  & 8\%   & 900   & 8.50\%   & 272.69 \\
70  & 7\%   & 1260  & 4.81\%   & 382.76 \\
100 & 6.0\% & 1800  & 4.61\%   & 544.51 \\
150 & 6.5\% & 2406  & 4.70\%   & 928.81 \\
200 & 7.5\% & 3132  & 4.67\%   & 1230.19 \\
\hline

\multicolumn{5}{|c|}{\text{Function Type: CIFAR-10(CNN)}} \\
\hline
10  & 35\%  & 360   & 26.19\%  & 60.11 \\
20  & 25\%  & 720   & 24.86\%  & 124.78 \\
30  & 12\%  & 1080  & 9.21\%   & 189.86 \\
50  & 10\%  & 1800  & 8.80\%   & 305.15 \\
70  & 9\%   & 2520  & 8.80\%   & 408.00 \\
100 & 8\%   & 3600  & 4.46\%   & 576.88 \\
150 & 5.4\% & 4530  & 3.12\%   & 987.55 \\
200 & 3.7\% & 5544  & 3.08\%   & 1265.01 \\
\hline

\end{tabular}
\end{center}

\caption{Comparison of DSE and MOAOCFGD for different function types.}
\label{Table 2}
\end{table}
\begin{figure}[!htbp]
  \centering
  \begin{subfigure}[b]{0.45\textwidth}
   \includegraphics[width=\linewidth,height=2.75cm]{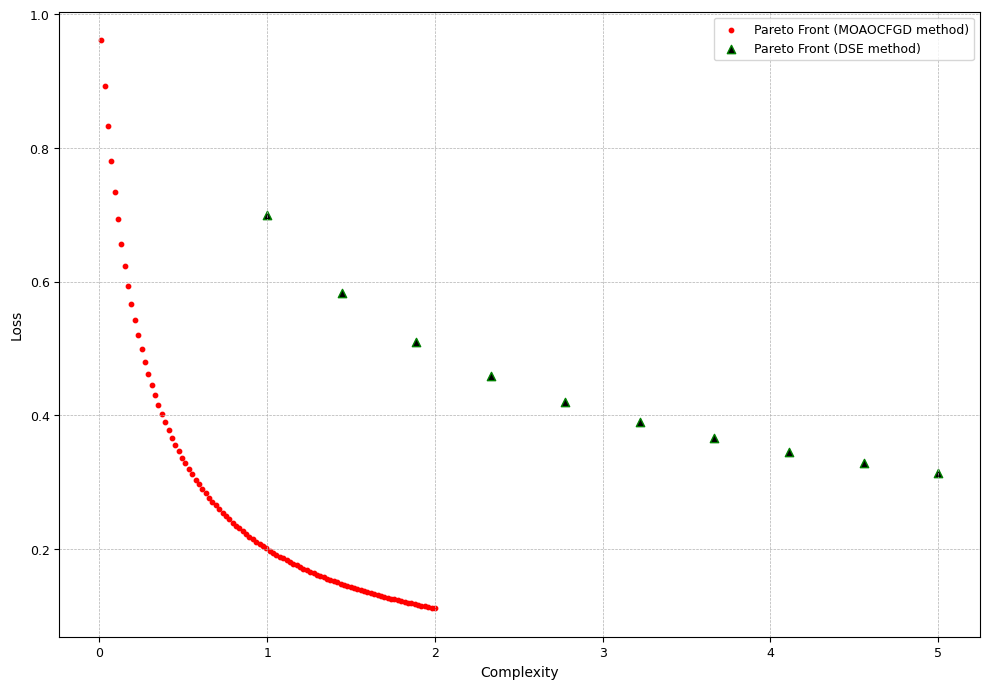}
    \label{fig31}
  \end{subfigure}
  \hfill
    \begin{subfigure}[b]{0.45\textwidth}
    \includegraphics[width=\linewidth,height=2.75cm]{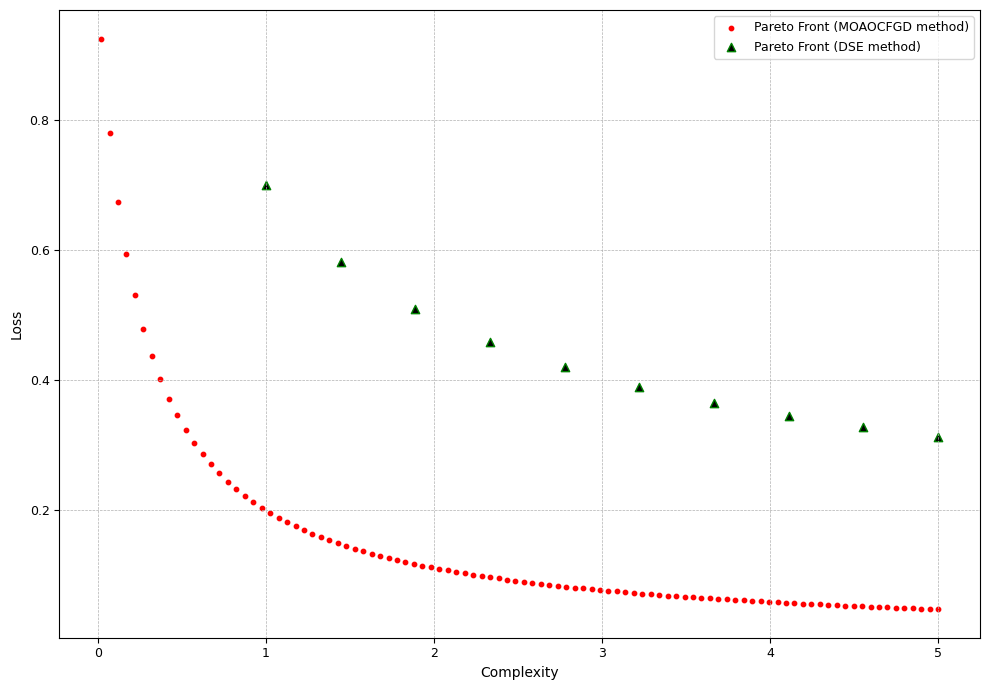}
    \label{fig32}
  \end{subfigure}
    \caption{Restricted MNIST and MNIST datasets of neural networks based on two objectives.}
  \label{fig1}

\end{figure}
\begin{figure}[!htbp]
  \centering
  \begin{subfigure}[b]{0.45\textwidth}
   \includegraphics[width=\linewidth,height=2.75cm]{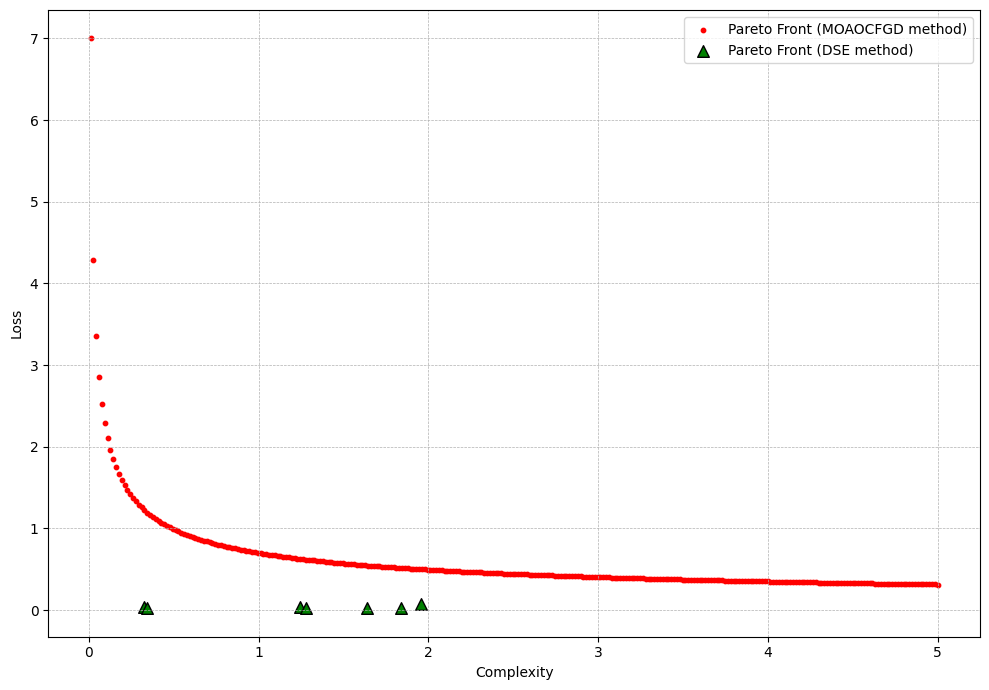}
    \label{fig41}
  \end{subfigure}
  \hfill
    \begin{subfigure}[b]{0.45\textwidth}
    \includegraphics[width=\linewidth,height=2.75cm]{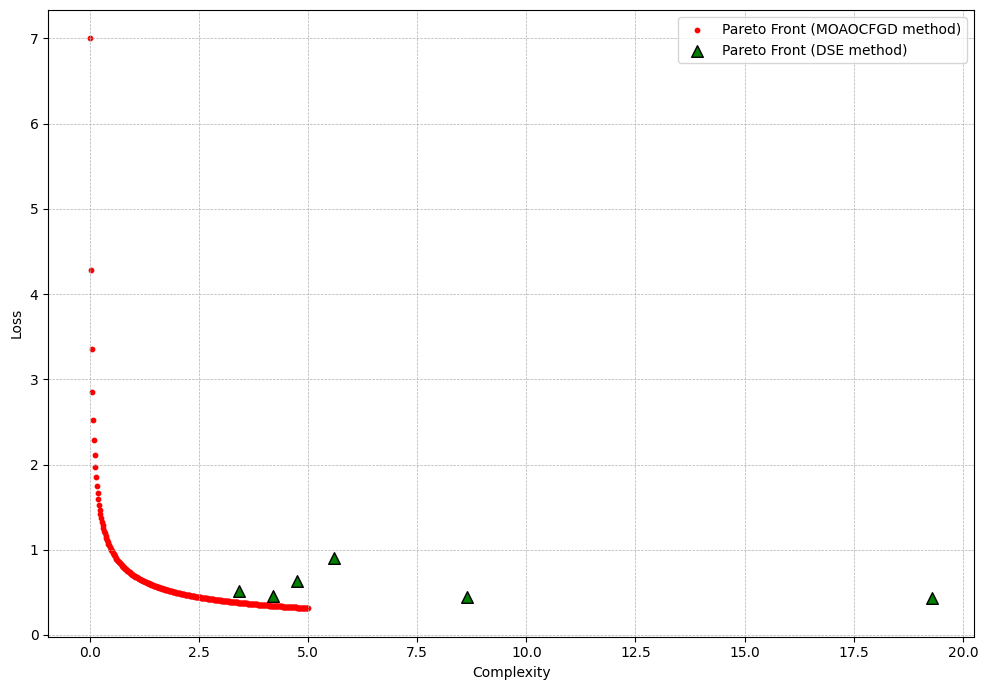}
    \label{fig42}
  \end{subfigure}
    \caption{MNIST(CNN) and CIFAR-10(CNN) datasets of neural networks based on two objectives.}
  \label{fig2}

\end{figure}

\noindent
\section{Conclusions}\label{sec6}
A new algorithm, MOAOCFGD, is presented for unconstrained $(MOP)$ with non-smooth functions. The Caputo fractional derivatives are used for their clear handling of constants. No weight functions or objective ordering are needed, unlike scalarization methods. A valid descent direction is found by forming a quadratic model of the objectives. An Armijo-type line search is then applied to ensure a sufficient decrease. This guarantees a true descent step and supports convergence in the optimization process. The convergence to a Tikhonov-regularized solution is shown via least squares theory. Error bounds ensure convergence to an integer-order critical point under mild conditions. The method effectively solves simple test problems. Better trade-offs were obtained compared to standard MOGD by using Tikhonov regularization. Numerical experiments on neural network training confirm the strength of the method. Tables and Pareto front plots highlight its superior performance. The fractional-derivative framework may be extended to other multi-objective methods, such as quasi-Newton or Newton-type approaches.  
\section*{Declarations}
\begin{itemize}
    \item {\bf Acknowledgement:} Author Barsha Shaw appreciates the financial support received from the University Grants Commission (UGC) Fellowship (ID No. 211610020050), Government of India, which supported her Ph.D. work.
    \item {\bf Ethics approval and consent to participate:} Not applicable.\\
\item {\bf Consent for publication:} The authors declare no consent for publication.\\
\item {\bf Data availability statement:} This paper does not involve any associated data.\\
\item {\bf Competing interest:} The author declares no conflict of interest.\\
\item {\bf Funding information:} No\\
\item {\bf Authors' contributions:} Both authors have equally contributed to the theoretical as well as numerical part.\\
\end{itemize}

\end{document}